\documentclass[11pt]{amsart}

\usepackage{amsfonts,amssymb,amsthm,eucal, mathtools,color,bbm,verbatim,graphicx,hyperref, enumitem}
\usepackage[french, english]{babel}

\usepackage[margin=1.25in]{geometry}
\calclayout

\newtheorem{thm}{Theorem}
\newtheorem{cor}[thm]{Corollary}
\newtheorem{lemma}[thm]{Lemma}

\theoremstyle{definition}

\newcommand{\E}{\mathbb{E}}

\newcommand{\N}{\mathbb{N}}
\newcommand{\n}{\mathcal{N}}

\newcommand{\C}{\mathbb{C}}
\renewcommand{\P}{\mathbb{P}}

\newcommand{\abs}[1]{\left\vert #1 \right\vert}

\DeclareMathOperator{\var}{Var}

\DeclarePairedDelimiter\ceil{\lceil}{\rceil}

\newcommand{\todo}[1]{\textsf{\color{blue} (#1)}}

\newcommand{\Unitary}[1]{\mathbb{U}\left(#1\right)}

\newcommand{\ind}[1]{\mathbbm{1}_{#1}}

\allowdisplaybreaks[2]
\thanks{\footnotemark {$^\dagger$} Supported in part by NSF DMS 1612589.}
\author{Elizabeth Meckes{$^\dagger$}}

\author{Kathryn Stewart{$^\dagger$}}

\address{Department of Mathematics, Applied Mathematics, and
  Statistics, Case Western Reserve University, 10900 Euclid Ave.,
  Cleveland, Ohio 44106, U.S.A.}

\email{elizabeth.meckes@case.edu}

\address{Department of Mathematics, Applied Mathematics, and
  Statistics, Case Western Reserve University, 10900 Euclid Ave.,
  Cleveland, Ohio 44106, U.S.A.}

\email{kathrynstewart@case.edu}

\title{Eigenvalue rigidity for truncations of random unitary matrices}

\begin{document}

\maketitle

\begin{abstract}
We consider the empirical eigenvalue distribution of an $m\times m$
principal submatrix of an $n\times n$ random unitary matrix
distributed according to Haar measure.  For $n$ and $m$ large with $\frac{m}{n}=\alpha$, the empirical spectral measure is well-approximated by a deterministic measure $\mu_\alpha$ supported on the unit disc.  In earlier work, we 
showed that for fixed $n$ and $m$, the bounded-Lipschitz distance
between the empirical spectral measure and the corresponding $\mu_\alpha$ is typically of order $\sqrt{\frac{\log(m)}{m}}$ or
smaller.  In this paper, we consider eigenvalues on a microscopic
scale, proving concentration inequalities for the eigenvalue counting
function and for individual bulk eigenvalues.

\end{abstract}

\section{Introduction}

Let $U$ be an $n \times n$ Haar-distributed unitary matrix and let
$U_m$ be the $m \times m$ top-left block of $U$, where $m < n$. We
refer to $U_m$ as a truncation of $U$. The eigenvalues of the
truncation are all located within the unit disc and the asymptotic
distribution of the eigenvalues can be described quite explicitly. Let  $\mu_{m}$ denote the empirical spectral measure of $U_m$, that is,
\begin{align*}
\mu_{m} = \frac{1}{m} \sum_{p=1}^m \delta_{\lambda_p},
\end{align*}
where $\lambda_1, \ldots , \lambda_m$ are the eigenvalues of $U_m$. Petz and R\'effy \cite{PR} proved that if $\frac{m}{n}\to\alpha\in(0,1)$, then $\mu_m$ converges almost surely to a limiting spectral measure $\mu_{\alpha}$; it has radial density with respect to Lebesgue measure on $\mathbb{C}$ given by 
\begin{equation*}
f_{\alpha}(z)= \begin{cases} \frac{(1-\alpha)}{\pi \alpha (1-|z|^2)^2}, & 0 < |z| < \sqrt{\alpha}; \\
0, & \mbox{otherwise}. \end{cases}
\end{equation*}

In \cite{MS19}, we proved the following non-asymptotic, quantitative version of this
result.  The rescaling was chosen so that the support of the limiting measure is the full unit disc, independent of $\alpha$.
\begin{thm} [E.\ Meckes and K.\ Stewart]\label{T:spectral approx}
Let
$n,m\in\N$ with $1\le m<n$. Let
$U\in\Unitary{n}$ be distributed according to Haar measure, and let
$\lambda_1, \ldots,
\lambda_m$ denote the eigenvalues of the top-left $m\times m$ block of
$\sqrt{\frac{n}{m}}U$.  The joint law of $\lambda_1,\ldots,\lambda_m$ is denoted
$\P_{n,m}$. 
Let $\mu_m$ be the random measure with mass $\frac{1}{m}$ at each of
the $\lambda_p$, and let $\alpha=\frac{m}{n}$. Let $\mu_\alpha$ be the probability
measure on the unit disc with the density $g_\alpha$ defined by
\begin{equation*}
g_{\alpha}(z)= \begin{cases} \frac{(1-\alpha)}{\pi (1-\alpha|z|^2)^2}, & 0 < |z| < 1; \\
0, & \mbox{otherwise}. \end{cases}
\end{equation*}
For any $r>0$,
\begin{equation*}
\mathbb{P}_{n,m} \Big[ d_{BL} \left( \mu_m , \mu_{\alpha} \right) \geq r \Big]  \le e^2\exp\left\{-C_\alpha m^2r^2+2m\log(m)+C_\alpha'm\right\}+\frac{e}{2\pi}\sqrt{\frac{m}{1-\alpha}}e^{-m},\end{equation*}
where $C_\alpha=\frac{1}{128\pi(1+\sqrt{3+\log(\alpha^{-1})})^2}$ and $C_\alpha'=6+3\log(\alpha^{-1})$. \\
\end{thm}

The result above is essentially macroscopic; it says that with high
probability, $d_{BL}(\mu_m,\mu_\alpha)$ is of order
$\sqrt{\frac{\log(m)}{m}}$. The purpose of this paper is to examine
the microscopic level, by considering the eigenvalue counting function on small sets. Throughout the paper, we assume that $\alpha=\frac{m}{n}$ is bounded away from 0 and 1; i.e., that there is a fixed $\delta>0$ such that $\alpha\in(\delta,1-\delta)$.  Throughout the statements and proofs, there are constants $C_\alpha$ depending only on $\alpha$; their exact values may vary from one line to the next.

\medskip

We begin by ordering the
eigenvalues $\left\{ \lambda_p \right\}_{p=1}^m$ in the spiral fashion
introduced in \cite{MM15}. Define a linear order $\prec$ on $\mathbb{C}$ by making 0 initial, and for nonzero $w,z \in \mathbb{C},$ declare $w \prec z$ if either of the following hold:
\begin{itemize}
    \item $|w| < |z|$ 
    \item $ |w| =  |z|$ and $\arg w < \arg z$
\end{itemize}
We divide the disc of radius $\sqrt{\frac{m}{n}}$ (i.e., the support of the limiting eigenvalue density) into annuli with radii $r_i=\frac{i}{\sqrt{n-m+i^2}}$; it is verified below that the expected number of eigenvalues in the annulus from radius $r_{i-1}$ to $r_i$ is approximately $2i-1$.

More generally, for $\theta\in(0,2\pi]$, define
\begin{align*}
    A_{i,\theta} & = \left\{ z \in \mathbb{C} \bigg| z \prec r_ie^{i\theta} \right\} \\
    & = \left\{ z \in \mathbb{C} \bigg| |z| < r_i \right\} \cup \left\{ z \in \mathbb{C} \bigg| r_i \leq |z| < r_{i+1}, 0 < \arg z < \theta \right\},
\end{align*}
with $r_i = \frac{i}{\sqrt{n-m+i^2}}$ and
$1 \leq i \leq \sqrt{m}$ (see Figure \ref{F:segment}).

\begin{figure} 
  \begin{centering}
    \includegraphics[width=3in]{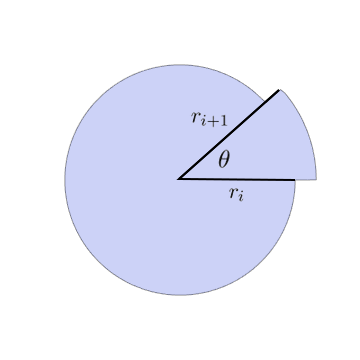}
  \end{centering}
  \caption{}
  \label{F:segment}
\end{figure}

Our first main result is on the concentration of the eigenvalue counting function for the sets $A_{i,\theta}$.
\begin{thm} \label{L:bernstein} Let $\mathcal{N}_{i,\theta}$ denote the number of eigenvalues of an $m\times m$ truncation of a Haar-distributed matrix in $\Unitary{n}$ which lie in $A_{i,\theta}$.  
If $\epsilon_m=\sqrt{\frac{2\log (m+1)}{m}}$, then for each $1 \leq i \leq \sqrt{m} \left( 1 - \frac{ \epsilon_m}{1 - \alpha\left(1-\epsilon_m \right)}\right)^{\frac{1}{2}}$, $0 \leq \theta \leq 2\pi$, and $t>0$, 
\begin{align*}
    \mathbb{P} &\left[ \left|\mathcal{N}_{i,\theta} - i^2 - \frac{\theta}{2\pi}(2i+1) \right|\geq t \right]  \leq 2e^{2}\exp\left[- \min \left\{ \tfrac{t^2}{C_\alpha i\sqrt{\log(i)}}, \tfrac{t}{4} \right\} \right].
\end{align*}
If  $t> \frac{12}{1-\alpha}\sqrt{2m \log (m+1)},$ then this estimate
is also valid for those $i$ with $ \sqrt{m} \left( 1 - \frac{ \epsilon_m}{1 - \alpha\left(1-\epsilon_m \right)}\right)^{\frac{1}{2}}\le i\le\sqrt{m}$.

\end{thm}

We next define predicted locations $\{ \tilde{\lambda}_p
\}_{p=1}^m$ for the eigenvalues by choosing $2i-1$ equally spaced points in the annulus with inner radius $r_{i-1}$ and outer radius $r_i$. 
The concentration inequalities in Theorem \ref{L:bernstein} for the counting function lead to the following concentration inequality for bulk eigenvalues about their predicted locations. 
\begin{thm} \label{L: individual}
Let $\{\lambda_p\}_{p=1}^m$ denote the eigenvalues of an $m\times m$ truncation of a Haar-distributed matrix in $\Unitary{n}$, ordered according to $\prec$.  Let $l=\ceil{\sqrt{p}}$.  There  are constants $c_\alpha,C_\alpha$ depending only on $\alpha=\frac{m}{n}$  such that, if $\epsilon_m=\sqrt{\frac{2\log (m+1)}{m}}$, then for those $p$ with 
\begin{align*}
2\le l\leq \sqrt{m}\left( 1 - \frac{ \epsilon_m}{1-\alpha ( 1- \epsilon_m)} \right)^{\frac{1}{2}}, 
\end{align*}
when $ s \leq 2\pi (l-1)$, 
\begin{align*}
    \mathbb{P}\left[ |\lambda_p - \tilde{\lambda}_p| \geq \tfrac{s}{\sqrt{n-m+(l-1)^2}} \right] \leq  2\exp\left[ - \frac{s^2}{C_\alpha l\sqrt{\log(l)}} \right];
\end{align*}
when $ 2\pi (l-1)< s \leq 2\sqrt{n-m+(l-1)^2}$,
\begin{align*}
    \mathbb{P}\left[ |\lambda_p - \tilde{\lambda}_p| \geq \tfrac{s}{\sqrt{n-m+(l-1)^2}} \right] \leq  2 \exp \left[ - c_\alpha s^2 \right];
\end{align*}
and when $s > 2\sqrt{n-m+(l-1)^2}$,
\begin{align*}
    \mathbb{P}\left[ |\lambda_p - \tilde{\lambda}_p| \geq \tfrac{s}{\sqrt{n-m+(l-1)^2}} \right] = 0.
\end{align*}
\end{thm}
By way of example, if $2\pi(l-1)\le \sqrt{\frac{k}{c_\alpha}\log(n)}$, then 
\begin{align*}
\mathbb{P}\left[ |\lambda_p - \tilde{\lambda}_p| \geq \sqrt{\frac{k\log(n)}{c_\alpha(n-m+(l-1)^2)}}\right] \leq 2n^{-k},
\end{align*} 
whereas if, e.g., $\log(n)\le\frac{4\pi^2(l-1)^2}{kC_\alpha l\sqrt{\log(l)}}$, then
\begin{align*}
\mathbb{P}\left[ |\lambda_p - \tilde{\lambda}_p| \geq (\log(l))^{\frac{1}{4}}\sqrt{\frac{kC_\alpha l\log(n)}{n-m+(l-1)^2}}\right] \leq 2n^{-k},
\end{align*} 

For reference, spacing of predicted locations around
$\tilde{\lambda}_p$ is about $\frac{1}{\sqrt{n-m+(l-1)^2}}$.

The concentration inequalities of Theorem \ref{L: individual} also
easily imply the following variance bound for bulk eigenvalues.

\begin{cor}\label{T:location-variance}
Let $\epsilon_m = \sqrt{\frac{2 \log (m+1)}{m}}$ and $p$ be such that $2\le \lceil\sqrt{p}\rceil \leq \sqrt{m}\left( 1 - \frac{ \epsilon_m}{1-\alpha ( 1- \epsilon_m)} \right)^{\frac{1}{2}}. $

There is a constant $C_\alpha$ depending only on $\alpha=\frac{m}{n}$ such that
\[\var \left( \lambda_p \right) \leq C_\alpha \frac{\sqrt{p\log(p+1)}}{n}.\]
\end{cor}

\bigskip

\section{Means and Variances} \label{S: means}
Throughout the proofs, we will make heavy use of the fact that the
eigenvalues we consider are a
determinantal point process on $\{|z|\le 1\}$ with kernel (with respect to
Lebesgue measure) given by
\begin{equation}\label{E:kernel-formula}
    K(z_1, z_2)= \sum_{j=1}^m \frac{1}{N_j} (z_1\overline{z_2})^{j-1}(1-|z_1|^2)^{\frac{n-m-1}{2}}(1-|z_2|^2)^{\frac{n-m-1}{2}},
\end{equation} 
with 
\begin{equation*}
    N_j = \frac{\pi (j-1)!(n-m-1)!}{(n-m+j-1)!}.
\end{equation*}
See, e.g., \cite{ZS} or \cite{PR}.

Recall that for large $n$ and $\frac{m}{n} = \alpha \in (0,1)$ the spectral measure of the truncation is approximately given by the measure $\mu_{\alpha}$, with density  with respect to Lebesgue measure  given by
\begin{equation*}
f_{\alpha}(z)= \begin{cases} \frac{(1-\alpha)}{\pi \alpha (1-|z|^2)^2}, & 0 < |z| < \sqrt{\alpha}; \\
0, & \mbox{otherwise}. \end{cases}
\end{equation*}
In particular, given a set $A\subseteq\{|z|\le\sqrt{\alpha}\}$, the
expected number $\mathcal{N}_A$ of eigenvalues inside $A$ is approximately
\(m\mu_\alpha(A)\). We begin by giving explicit
estimates quantifying this approximation.

\begin{lemma} \label{L: expectedval} 
For any measurable $A\subseteq\{|z|\le\sqrt{\alpha}\}$,
\[m\mu_\alpha(A)-\frac{6\sqrt{2m\log(m+1)}}{1-\alpha}\le\E\mathcal{N}_A\le m\mu_\alpha.\]

If additionally 
  $A\subseteq\left\{|z|^2 \le \alpha \left(1-\sqrt{\frac{2\log(m+1)}{m}}\right)\right\},$
then 
\[m\mu_\alpha(A)-4\le\E \mathcal{N}_A\le m\mu_\alpha(A).\]

\end{lemma}

\begin{proof}
For a determinantal point process on $(\Lambda,\mu)$ with kernel $K$,
the expected number of points in a set $A$ is given by 
\[\mathbb{E}\mathcal{N}(A)= \int_{A} K(x,x)d\mu(x).\] 

From the formula for the kernel given in equation \eqref{E:kernel-formula},
\begin{align*}
K(z,z)&=\frac{(1-|z|^2)^{n-m-1}}{\pi}\sum_{j=1}^{m}\frac{(n-m+j-1)!}{(j-1)!(n-m-1)!}|z|^{2(j-1)}\\&=\frac{(n-m)(1-|z|^2)^{n-m-1}}{\pi}\sum_{p=0}^{m-1}\frac{(n-m+1)_{p}}{p!}|z|^{2p},
\end{align*}
where $(a)_{p}=a(a+1)\cdots(a+p-1)$ is the rising Pochhammer symbol.
Letting  \[{}_2F_1(a,b;c;z)=\sum_{n=0}^\infty\frac{(a)_{n}(b)_{n}}{(c)_{n}}\frac{z^n}{n!}\]
denote the hypergeometric function with parameters $a,b,c$, 
\[\sum_{p=0}^{\infty}\frac{(n-m+1)_{p}}{p!}|z|^{2p}={}_2F_1(n-m+1,1;1;|z|^2)=(1-|z|^2)^{-n+m-1}.\]
It follows that 
\begin{equation}\begin{split}\label{E:Kzz-expansion}K(z,z)&=\frac{n-m}{\pi}\left[\frac{1}{(1-|z|^2)^2}-(1-|z|^2)^{n-m-1}\sum_{p=m}^\infty \frac{(n-m+1)_{p}}{p!}|z|^{2p}\right]\\&=mf_\alpha(z)\left[1-(1-|z|^2)^{n-m+1}\sum_{p=m}^\infty \frac{(n-m+1)_{p}}{p!}|z|^{2p}\right],\end{split}\end{equation}
and as an immediate consequence,
\[\E\mathcal{N}_{A}\le \int_{A}mf_\alpha(z)d\lambda(z)=m\mu_\alpha(A).\]

For the lower bound, we first treat the more restrictive case of
\begin{align*}
A\subseteq\left\{|z|^2 \le\alpha \left(1-\sqrt{\frac{2\log(m+1)}{m}}\right)\right\}.
\end{align*}
 Consider the random variable $Y_{k}(x)$ on $\N\cup\{0\}$ with mass function 
\[\P[Y_{k}(x)=p]=\frac{(k)_{p}}{p!}(1-x)^{k}x^p.\]
The moment generating function of $Y_k(x)$ is given by
\[\E[e^{tY_k(x)}]=\sum_{p=0}^\infty \frac{(k)_{p}}{p!}(1-x)^{k}(e^tx)^p=\left[\frac{1-x}{1-xe^t}\right]^k.\]
Now, 
\begin{align*}
(1-|z|^2)^{n-m+1}&\sum_{p=m}^\infty \frac{(n-m+1)_{p}}{p!}|z|^{2p}\\&=\P\left[Y_{n-m+1}(|z|^2)\ge m\right]\le e^{-t m}\left[\frac{1-|z|^2}{1-|z|^2e^t}\right]^{n-m+1},
\end{align*}
for any $t>0$.  Since
$|z|^2<\alpha \left(1-\sqrt{\frac{2\log(m+1)}{m}}\right)<\frac{\alpha n}{n+1}$,
we may choose $t=\log\left(\frac{m}{|z|^2(n+1)}\right)>0.$ 
Then
\begin{align*}
(1-|z|^2)^{n-m+1}&\sum_{p=m}^\infty \frac{(n-m+1)_{p}}{p!}|z|^{2p}\\&\le
\left(\frac{|z|^2(n+1)}{\alpha n}\right)^{\alpha n}\left[\frac{1-|z|^2}{1-\alpha\left(\frac{n}{n+1}\right)}\right]^{n(1-\alpha)+1}.
\end{align*}
For $|z|^2\le \frac{\alpha n}{n+1},$ this last quantity is increasing
in $|z|$; if we further assume that $|z|^2\le\alpha(1-\epsilon_n)$,
we thus have that 

\begin{align*}
(1- & |z|^2)^{n-m+1}\sum_{p=m}^\infty \frac{(n-m+1)_{p}}{p!}|z|^{2p}\\&\le
e^\alpha(1-\epsilon_n)^{\alpha n}\left[\frac{1-\alpha(1-\epsilon_n)}{1-\alpha\left(\frac{n}{n+1}\right)}\right]^{n(1-\alpha)+1}\\&=\exp\left\{\alpha+\alpha   n\log(1-\epsilon_n)+(n(1-\alpha)+1)\log\left(1+\frac{\alpha\left(\epsilon_n-\frac{1}{n+1}\right)}{1-\alpha\left(\frac{n}{n+1}\right)}\right)\right\}\\&\le\exp\left\{\alpha-\alpha
        n\epsilon_n-\frac{\alpha
                                                                                                                                                                                                                                                                                                                      n\epsilon_n^2}{2}+(n+1)\alpha\left(\epsilon_n-\frac{1}{n+1}\right)\right\}\\&=\exp\left\{\alpha\epsilon_n-\frac{\alpha n\epsilon_n^2}{2}\right\}.
\end{align*}
The claimed estimate follows by taking $\epsilon_n=\sqrt{\frac{2\log(m+1)}{m}}=\sqrt{\frac{2\log(\alpha
  n+1)}{\alpha n}}$ (the constant $4$ in the statement is for
concreteness; the actual estimate resulting from this choice of
$\epsilon_n$ is $e^{\sqrt{\frac{2\alpha\log(\alpha n+1)}{n}}}$).

Returning to the more general case, using the expression for $K(z,z)$
in \eqref{E:Kzz-expansion}
\begin{align*}\E\mathcal{N}_A&=m\mu_\alpha(A)-\int_Amf_\alpha(z)(1-|z|^2)^{n-m+1}\sum_{p=m}^\infty\frac{(n-m+1)_{p}}{p!}|z|^{2p}d\lambda(z)\\&\ge m\mu_\alpha(A)-4-\int_{A\cap\left\{\alpha\left(1-\sqrt{\frac{2\log(m+1)}{m}}\right)\le |z|^2\le\alpha\right\}}mf_\alpha(z)(1-|z|^2)^{n-m+1} \\
                                                                                                                                            & \qquad \qquad \qquad \qquad \qquad \qquad \qquad\qquad\times \sum_{p=m}^\infty\frac{(n-m+1)_{p}}{p!}|z|^{2p}d\lambda(z),\end{align*}
making use of the analysis above.  To estimate the remaining integral,
we reconsider the quantity 
\[\P[Y_{n-m+1}(|z|^2)\ge m],\]
this time simply estimating via Markov's inequality.  Given $k$ and
$x$,
\begin{align*}
\E Y_k(x)=\sum_{p=0}^\infty p \frac{(k)_{p}}{p!}(1-x)^{k}x^p=\frac{kx}{1-x}\sum_{\ell=0}^\infty\frac{(k+1)_{\ell}}{\ell!}(1-x)^{k+1}x^\ell=\frac{xk}{1-x},
\end{align*}
and so
\begin{align*}
(1-|z|^2)^{n-m+1}&\sum_{p=m}^\infty \frac{(n-m+1)_{p}}{p!}|z|^{2p}\\&=\P\left[Y_{n-m+1}(|z|^2)\ge m\right]\le \frac{(n-m+1)|z|^2}{m(1-|z|^2)}.
\end{align*}
It follows that 
\begin{align*}
&\int_{A\cap\left\{\alpha\left(1-\sqrt{\frac{2\log(m+1)}{m}}\right)\le
      |z|^2\le\alpha\right\}}mf_\alpha(z)(1-|z|^2)^{n-m+1} \textstyle \sum_{p=m}^\infty\frac{(n-m+1)_{p}}{p!}|z|^{2p}d\lambda(z)\\& \qquad\le
\int_{\left\{\alpha\left(1-\sqrt{\frac{2\log(m+1)}{m}}\right)\le
      |z|^2\le\alpha\right\}}f_\alpha(z)\frac{(n-m+1)|z|^2}{(1-|z|^2)}d\lambda(z)
\\& \qquad\le (n-m+1)\left(\sup_{\left\{\alpha\left(1-\sqrt{\frac{2\log(m+1)}{m}}\right)\le
      |z|^2\le\alpha\right\}}\frac{|z|^2f_\alpha(z)}{(1-|z|^2)}\right)\pi\alpha\sqrt{\frac{2\log(m+1)}{m}}.
\end{align*}

Now,
\[\frac{|z|^2f_\alpha(z)}{1-|z|^2}=\frac{(1-\alpha)|z|^2}{\pi\alpha(1-|z|^2)^3}\le\frac{1}{\pi(1-\alpha)^2}\]
for $|z|^2\le\alpha$,
and so 
\begin{align*}
(n-m+1)&\left(\sup_{\left\{\alpha\left(1-\sqrt{\frac{2\log(m+1)}{m}}\right)\le
      |z|^2\le\alpha\right\}}\frac{|z|^2f_\alpha(z)}{(1-|z|^2)}\right)\pi\alpha\sqrt{\frac{2\log(m+1)}{m}} \\&\le (n-m+1)\tfrac{\alpha}{(1-\alpha)^2}\sqrt{\tfrac{2\log(m+1)}{m}}\\&=\tfrac{1}{1-\alpha}\sqrt{2m\log(m+1)}+\tfrac{\alpha}{(1-\alpha)^2}\sqrt{\tfrac{2\log(m+1)}{m}}.
\end{align*}

It follows that 
\begin{align*}
\E\mathcal{N}_A&\ge m\mu_\alpha(A)-4-\tfrac{1}{1-\alpha}\sqrt{2m\log(m+1)}-\tfrac{\alpha}{(1-\alpha)^2}\sqrt{\tfrac{2\log(m+1)}{m}}.
\end{align*}
Observing that $\max\left\{4,\frac{\alpha}{(1-\alpha)^2}\sqrt{\frac{2\log(m+1)}{m}}\right\}\le\frac{1}{1-\alpha}\sqrt{2m\log(m+1)}$ completes the proof.
\end{proof}

The following is an immediate consequence of Lemma \ref{L:
  expectedval}.

\begin{cor}\label{T:segment-means}
Let $\theta\in(0,2\pi]$ and let $r_i=\frac{i}{\sqrt{n-m+i^2}}$. Let $\mathcal{N}_{i,\theta}$ be defined as above, and suppose that $m \geq 3$ and  $1\le i\le \sqrt{m} \left( 1 - \frac{ \sqrt{\frac{2\log (m+1)}{m}}}{1 - \alpha + \alpha \sqrt{\frac{2 \log (m+1)}{m}}} \right)$. Then
\begin{equation*}
     \left|\E\mathcal{N}_{i,\theta}-i^2-\frac{\theta}{2\pi}\big(2i+1\big)\right|\le 4.
\end{equation*}

\end{cor}

\begin{proof}
The condition on $i$ guarantees that $A_{i,\theta}\subseteq\left\{|z|^2\le\alpha \left(1-\sqrt{\frac{2\log(m+1)}{m}}\right)\right\},$ so that the sharper estimate from Lemma \ref{L: expectedval} applies.

For $A_{i,\theta}$ defined as above,
\begin{align*}
\mu_{\alpha} \left( A_{i,\theta} \right) & = \int_0^{2\pi} \int_0^{r_i} \frac{1-\alpha}{\pi \alpha (1-r^2)^2} r drd\theta + \int_0^{\theta} \int_{r_i}^{r_{i+1}} \frac{1-\alpha}{\pi \alpha (1-r^2)^2} r drd\theta \\
& = \frac{1-\alpha}{\alpha} \frac{r_i^2}{1-r_i^2} + \frac{(1-\alpha)\theta}{\alpha 2\pi} \frac{r_{i+1}^2-r_i^2}{\left( 1-r_i^2 \right) \left( 1 - r_{i+1}^2 \right)},
\end{align*}
so that 
\begin{align*}
m\mu_\alpha \left( A_{i,\theta} \right) &=(n-m) \frac{r_i^2}{1-r_i^2} +\frac{(n-m)\theta}{2\pi}\frac{r_{i+1}^2-r_i^2}{\left( 1-r_i^2 \right) \left( 1 - r_{i+1}^2 \right)}\\&=i^2+\frac{\theta}{2\pi}\big(2i+1\big).
\end{align*}
\end{proof}

\medskip

We next estimate the variance of $\mathcal{N}_{i,\theta}$.

\begin{lemma} \label{L: variance}
Let $A_{i,\theta}$ be as above.  There is a constant $C_\alpha$ depending only on $\alpha=\frac{m}{n}$ such that
\begin{equation*}
\var(\n_{i,\theta}) \leq C_\alpha i\sqrt{\log(i)}.
\end{equation*}
\end{lemma}
\begin{proof}
By an argument similar to the one in \cite[Appendix B]{Gus}, 
  \begin{equation}
    \label{E:variance}
    \begin{split}
      \var(\n_{i,\theta})
      &= \int\limits_{\{\abs{z} <r_i\}} \int\limits_{\{\abs{w} \ge r_{i+1}\}}
      \abs{K(z,w)}^2 \ dw \ dz \\
      & \quad + \int\limits_{\{\abs{z} <r_i\}} \int\limits_{\{r_i\le \abs{w} <r_{i+1},\ \theta\le \arg w \le 2\pi\}} 
      \abs{K(z,w)}^2 \ dw \ d z \\
      & \quad + \int\limits_{\{r_i\le \abs{z} < r_{i+1},\ 0<\arg z  \le \theta\}}
      \int\limits_{\{\abs{w} \ge r_{i+1}\}} \abs{K(z,w)}^2 \ dw \ dz \\
      & \quad + \int\limits_{\{r_i\le \abs{z} < r_{i+1},\ 0<\arg z  \le \theta\}}
      \int\limits_{\{r_i \le \abs{w} <r_{i+1}, \ 0<\arg w \ge \theta\}} \abs{K(z,w)}^2 \ dw \ dz \\&=:V_1+V_2+V_3+V_4
    \end{split} 
  \end{equation}
  Observe that for $r_1,r_2\le 1$,
  \begin{align*}
  &\left|K(r_1 e^{i\varphi_1}, r_2 e^{i \varphi_2})\right|^2\\&\qquad = \frac{(n-m)^2}{\pi^2}(1-r_1^2)^{n-m-1}(1-r_2^2)^{n-m-1}\\&\qquad\qquad\qquad\qquad\sum_{j,k=0}^{m-1}\binom{n-m+j}{j}\binom{n-m+k}{k}(r_1r_2)^{j+k}e^{i(j-k)(\varphi_1-\varphi_2)}.
  \end{align*}

Integrating in polar coordinates gives that 
\begin{align}\label{E:I1}
\begin{split}
V_1& =4(n-m)^2 \\
& \quad \times \sum \limits_{j=0}^{m-1}\binom{n-m+j}{j}^2\int_0^{r_i} (1-r^2)^{n-m-1}r^{2j+1}dr\int_{r_{i+1}}^1 (1-r^2)^{n-m-1}r^{2j+1}dr,
\end{split}
\end{align}
since the angular integrals vanish unless $j=k$.  
By repeated applications of integration by parts, 
\begin{align*} 2(n-m)&\binom{n-m+j}{j}\int_{0}^{r_i}
                       (1-r^2)^{n-m-1}r^{2j+1}dr\\&=\sum_{\ell=j+1}^{n-m+j}\binom{n-m+j}{\ell}r_{i}^{2\ell}(1-r_{i}^2)^{n-m+j-\ell},
                                                    \end{align*}
which is exactly $\P[Y_j> j]$ for
$Y_j\sim\mathrm{Binom}(n-m+j,r_{i}^2).$ Similarly,
\begin{align*} 2(n-m)&\binom{n-m+j}{j}\int_{r_{i+1}}^1
                       (1-r^2)^{n-m-1}r^{2j+1}dr\\&=\sum_{\ell=0}^j\binom{n-m+j}{\ell}r_{i+1}^{2\ell}(1-r_{i+1}^2)^{n-m+j-\ell},
                                                    \end{align*}
which is $\P[X_j\le j]$ for
$X_j\sim\mathrm{Binom}(n-m+j,r_{i+1}^2).$ 

It thus follows from \eqref{E:I1} that
\begin{equation}\label{E:V1}V_1=\sum_{j=0}^{m-1}\P[Y_j>j]\P[X_j\le j]\le\sum_{j=0}^{i^2-1}\P[X_j\le j]+\sum_{j=i^2}^{m-1}\P[Y_j>j].\end{equation}

For the first sum, observe that $\E X_j=\frac{(n-m+j)(i+1)^2}{n-m+(i+1)^2}>j$ for $j\le i^2-1$.  By Bernstein's inequality,
\begin{equation}\begin{split}\label{E:Binom-Bernstein-1}
\P[X_j\le j]&=\P\left[\E X_j-X_j\ge\frac{(n-m)((i+1)^2-j)}{n-m+(i+1)^2}\right]\\&\le \exp\left\{-\min\left\{\frac{(n-m)((i+1)^2-j)^2}{2(n-m+j)(i+1)^2},\frac{(n-m)((i+1)^2-j)}{2(n-m+(i+1)^2)}\right\}\right\}.
\end{split}\end{equation}

The first term of the minimum is smaller exactly when $j\ge j_0:=\frac{(i+1)^4}{n-m+2(i+1)^2}$.  Note that $j_0\le \frac{(i+1)^2}{2}$, so that 
\begin{align*}\sum_{j=0}^{j_0}\exp\left\{-\frac{(n-m)((i+1)^2-j)}{2(n-m+(i+1)^2)}\right\}&\le\frac{(i+1)^2}{2}\exp\left\{-\frac{(n-m)(i+1)^2}{4(n-m+(i+1)^2)}\right\}\\&\le\frac{(i+1)^2}{2}\exp\left\{-\frac{(1-\alpha)(i+1)^2}{4}\right\},\end{align*}
which is bounded independent of $i$.

Now consider
\begin{align*}\sum_{j=j_0}^{i^2-1}&\P[X_j\le j]\\&\le\sum_{j=j_0}^{(i+1)^2-(i+1)\sqrt{\frac{2\log(i+1)}{1-\alpha}}}\exp\left\{-\frac{(n-m)((i+1)^2-j)^2}{2(n-m+j)(i+1)^2}\right\}+(i+1)\sqrt{\frac{2\log(i+1)}{1-\alpha}}\\&\le (i+1)+\frac{(i+1)\sqrt{2\log(i+1)}}{\sqrt{1-\alpha}},\end{align*}
where we have used the fact that the summand in the second line is increasing in $j$ and bounded by $\frac{1}{i+1}$ at the upper limit of the sum.

For the second sum of Equation \eqref{E:V1}, we again apply Bernstein's inequality: 
\begin{equation}\begin{split}\label{E:Binom-Bernstein-2}
\P[Y_j> j]&=\P\left[Y_j-\E Y_j> j-(n-m+j)r_{i}^2\right]\\&\le \exp\left\{-\min\left(\frac{(j-(n-m+j)r_{i}^2)^2}{2(n-m+j)r_{i}^2(1-r_{i}^2)},\frac{j- (n-m+j)r_{i}^2}{2}\right)\right\}\\&=\exp\left\{-\min\left\{\frac{(n-m)(j-i^2)^2}{2(n-m+j)i^2},\frac{(n-m)(j-i^2)}{2(n-m+i^2)}\right\}\right\}.
\end{split}\end{equation}

The change in behavior of the bound is at $j=j_1:=\frac{i^2[2(n-m)+i^2]}{n-m}$.  Note that $j_1\le i^2\left(2+\frac{\alpha}{1-\alpha}\right)$ since $i^2\le m$.  Decomposing as before,
\begin{align*}
\sum_{j=i^2}^{m-1}\P[Y_j>j]\le & i\sqrt{\frac{2\log(i)}{1-\alpha}}+\sum_{j=i^2+i\sqrt{\frac{2\log(i)}{1-\alpha}}}^{j_1}\exp\left\{-\frac{(n-m)(j-i^2)^2}{2(n-m+j)i^2}\right\}\\&\qquad +\sum_{j=j_1}^{m-1}\exp\left\{-\frac{(n-m)(j-i^2)}{2(n-m+i^2)}\right\}\\&\le i\sqrt{\frac{2\log(i)}{1-\alpha}}+i\left(2+\frac{\alpha}{1-\alpha}\right) +\sum_{j=j_1}^{m-1}\exp\left\{-\frac{(n-m)(j-i^2)}{2(n-m+i^2)}\right\}.
\end{align*}

This last sum is
\begin{align*} \begin{split}
e^{\frac{(n-m)i^2}{2(n-m+i^2)}}  \sum_{j=j_1}^{m-1} e^{-\frac{(n-m)j}{2(n-m+i^2)}} &\le e^{\frac{(n-m)i^2}{2(n-m+i^2)}} \left[\frac{e^{-\frac{(n-m)j_1}{2(n-m+i^2)}}}{1-e^{-\frac{(n-m)}{2(n-m+i^2)}}}\right]=\frac{e^{\frac{-i^2(n-m+i^2)}{2(n-m+i^2)}}}{1-e^{-\frac{(n-m)}{2(n-m+i^2)}}},
\end{split} \end{align*}
which is bounded independent of $i$.
Collecting terms, we have 
\[V_1\le C_\alpha i\sqrt{\log(i)}\]
for a constant $C_\alpha$ depending only on $\alpha$.

The remaining terms of \eqref{E:variance} are estimated similarly.  For $V_2$, integrating in polar coordinates gives that
\begin{align*}
V_2&=4(n-m)^2\left(1-\frac{\theta}{2\pi}\right)\sum_{j=0}^{m-1}\binom{n-m+j}{j}^2\\
& \quad \times \int_0^{r_i}(1-r^2)^{n-m-1}r^{2j+1}dr\int_{r_i}^{r_{i+1}}(1-r^2)^{n-m-1}r^{2j+1}dr\\
& \leq 	4(n-m)^2\left(1-\frac{\theta}{2\pi}\right)\sum_{j=0}^{m-1}\binom{n-m+j}{j}^2\\
& \quad \times \int_0^{r_i}(1-r^2)^{n-m-1}r^{2j+1}dr\int_{r_i}^{1}(1-r^2)^{n-m-1}r^{2j+1}dr.
\end{align*}
Proceeding exactly as for $V_1$,
\begin{align*} \begin{split}
V_2 & \leq \left( 1 - \frac{\theta}{2\pi} \right) \sum_{j=0}^{m-1} \mathbb{P} \left[Y_j > j \right] \mathbb{P} \left[ Y_j \leq j \right] \\
& \leq \left( 1 - \frac{\theta}{2\pi} \right) \left( \sum_{j=0}^{i^2-1} \mathbb{P} \left[ Y_j \leq j \right] + \sum_{j=i^2}^{m-1} \mathbb{P} \left[ Y_j > j \right] \right) \\
& \leq \left( 1 - \frac{\theta}{2\pi} \right) C_\alpha i\sqrt{\log(i)},
\end{split}\end{align*}
where $Y_j \sim$ Binom$(n-m+j, r_i^2)$. 

Integrating in polar coordinates and proceeding as above,
\begin{align*}
V_3&=4(n-m)^2\tfrac{\theta}{2\pi}\textstyle \sum_{j=0}^{m-1}\tbinom{n-m+j}{j}^2 \int_{r_i}^{r_{i+1}}(1-r^2)^{n-m-1}r^{2j+1}dr\int_{r_{i+1}}^{1}(1-r^2)^{n-m-1}r^{2j+1}dr\\
	&\leq 4(n-m)^2\tfrac{\theta}{2\pi}\textstyle \sum_{j=0}^{m-1}\tbinom{n-m+j}{j}^2 \int_{0}^{r_{i+1}}(1-r^2)^{n-m-1}r^{2j+1}dr\int_{r_{i+1}}^{1}(1-r^2)^{n-m-1}r^{2j+1}dr\\
	& = \frac{\theta}{2\pi} \sum_{j=0}^{m-1} \mathbb{P} \left[X_j >j \right] \mathbb{P} \left[ X_j \leq j \right] \\
	& \leq \frac{\theta}{2\pi}C_\alpha i\sqrt{\log(i)},
	\end{align*}
	where $X_j \sim$ Binom$(n-m+j, r_{i+1}^2)$. 
The final integral in \eqref{E:variance} is
\begin{align*}
V_4&=\frac{(n-m)^2}{\pi^2}\sum_{j,k=0}^{m-1}\binom{n-m+j}{j}\binom{n-m+k}{k}\\
& \quad \times \left(\int_{r_i}^{r_{i+1}}(1-r^2)^{n-m-1}r^{j+k+1}dr\right)^2 \int_0^{\theta}e^{i(j-k)\phi}d\phi\int_{\theta}^{2\pi}e^{i(k-j)\phi}d\phi.
\end{align*}
For $j\neq k$,
\begin{align*}
\int_{\theta}^{2\pi}e^{i(k-j)\phi}d\phi = -\int_0^{\theta}e^{i(k-j)\phi}d\phi = - \overline{\int_0^{\theta}e^{i(j-k)\phi}d\phi}.
\end{align*}
Therefore, if $j\neq k$ in the sum the term is negative. Thus
\begin{align*}
V_4&\leq4(n-m)^2\left(1-\frac{\theta}{2\pi}\right)\frac{\theta}{2\pi}\sum_{j=0}^{m-1}\binom{n-m+j}{j}^2\\
& \quad \times \left(\int_{r_i}^{r_{i+1}}(1-r^2)^{n-m-1}r^{2j+1}dr\right)^2\\
& \leq4(n-m)^2\left(1-\frac{\theta}{2\pi}\right)\frac{\theta}{2\pi}\sum_{j=0}^{m-1}\binom{n-m+j}{j}^2 \\
& \quad \times \int_{0}^{r_{i+1}}(1-r^2)^{n-m-1}r^{2j+1}dr\int_{r_i}^{1}(1-r^2)^{n-m-1}r^{2j+1}dr\\
& = \left(1-\frac{\theta}{2\pi}\right)\frac{\theta}{2\pi} \sum_{j=0}^{m-1} \mathbb{P} \left[X_j > j \right] \mathbb{P} \left[Y_j \leq j \right] \\
& \leq \left(1-\frac{\theta}{2\pi}\right)\frac{\theta}{2\pi}C_\alpha i\sqrt{\log(i)}.
\end{align*}
All together then,
\begin{align*}
 \var(\n_{i,\theta})& = V_1+V_2+V_3+V_4\leq C_\alpha i\sqrt{\log(i)}
\end{align*}
for a constant $C_\alpha$ depending only on $\alpha$.
\end{proof}

\section{Concentration} \label{S: rigidity}

We now move on to 
concentration for the counting functions $\mathcal{N}_{i,\theta}$.
The key ingredient is the following general result on
determinantal point processes.
\begin{thm}[Hough--Krishnapur--Peres--Vir\'ag \cite{HKPV}]\label{T:HKPV}
Let $\Lambda$ be a locally compact Polish space and $\mu$ a Radon
measure on $\Lambda$.  Suppose that
$K:\Lambda\times\Lambda\to\C$ is the kernel of a determinantal point
process, such that the
corresponding integral operator $\mathcal{K}:L^2(\mu)\to L^2(\mu)$
defined by 
\[[\mathcal{K}f](x)=\int_\Lambda K(x,y)f(y)d\mu(y)\]
is self-adjoint, nonnegative, and locally trace-class.  Let $D\subseteq\Lambda$ be such that the
restriction $K_D(x,y)=\ind{D}(x)K(x,y)\ind{D}(y)$ defines a
trace-class operator $\mathcal{K}_D$ on $L^2(\mu)$.  Then the number of points $\mathcal{N}_D$ lying in $D$ of
the process governed by $K$ is distributed as $\sum_k\xi_k$, where the
$\xi_k$ are independent Bernoulli random variables whose means are
given by the eigenvalues of the operator $\mathcal{K}_D$.  
\end{thm}

It is not hard to see that the kernel given in \eqref{E:kernel-formula} has the properties required by
Theorem \ref{T:HKPV}, and so the random variable
$\mathcal{N}_{i,\theta}$ is distributed as a sum of independent
Bernoulli random variables.  It is thus an immediate consequence of
Bernstein's inequality that
\begin{equation}\label{E:Bernstein}
\P\left[\big|\mathcal{N}_{i,\theta}-\E\mathcal{N}_{i,\theta}\big|>t\right]\le 2\exp\left(-\min\left\{\frac{t^2}{2\sigma^2},\frac{t}{2}\right\}\right).
\end{equation}
This is the key observation underlying the proof of Theorem \ref{L:bernstein}.

\begin{proof}[Proof of Theorem \ref{L:bernstein}]
For the first claim, the assumption on $i$ implies that
\begin{align*}
A_{i, \theta} \subseteq \sqrt{\alpha} \left( 1 - \sqrt{ \frac{2 \log (m+1)}{m}} \right)^{\frac{1}{2}} D
\end{align*}
where $D$ is the unit disc, so that by lemmas \ref{L: expectedval} and \ref{L: variance} together with Bernstein's inequality, if $t>4$, then
\begin{align*} \begin{split}
\mathbb{P} \left[\left| m\mu_{\alpha} \left(A_{i,\theta}\right) - \mathcal{N} \left( A_{i, \theta} \right)\right| \geq t \right]  & \leq \mathbb{P} \left[ \left|\mathbb{E} \mathcal{N} \left( A_{i,\theta} \right) - \mathcal{N} \left( A_{i, \theta} \right)\right| \geq t -4 \right] \\
& \leq \exp \left[ - \min \left\{ \frac{(t-4)^2}{C_\alpha i\sqrt{\log(i)}}, \frac{t-4}{2} \right\} \right].
\end{split} \end{align*}
 If $t \geq 8$, then $t-4 \geq t/2$, so that
\begin{align*}
\mathbb{P} \left[\left| m\mu_{\alpha} \left(A_{i,\theta}\right) - \mathcal{N} \left( A_{i, \theta} \right)\right| \geq t \right] & \leq  \exp \left[ - \min \left\{ \frac{t^2}{C_\alpha i\sqrt{\log(i)})}, \frac{t}{4} \right\} \right];
\end{align*}
if $t < 8$, then 
\begin{align*}
\exp \left[ - \min \left\{ \frac{t^2}{C_\alpha i\sqrt{\log(i)}}, \frac{t}{4} \right\} \right] > e^{- 2}
\end{align*}
and the first claim follows. The proof of the second claim is an immediate consequence of the second estimate of Lemma \ref{L: expectedval} together with Lemma \ref{L: variance}.
\end{proof}

We now focus our attention on individual eigenvalues.  
 Given $1\leq p \leq m$, let $l=\ceil{\sqrt{p}}$ and $q= p
- (l-1)^2$, so that $p = (l-1)^2+q$ and $1\leq q \leq 2l-1$. Let 
\[r_l=\frac{l}{\sqrt{n-m+l^2}}.\]
The
predicted locations $\tilde{\lambda}_p$ for the eigenvalues are
defined by 
\begin{equation*}
    \tilde{\lambda}_p = r_{l-1}e^{2\pi i q/(2l-1) }=\frac{l-1}{\sqrt{n-m+(l-1)^2}}e^{2\pi i q/(2l-1) }.
\end{equation*}
To shed some light on this choice, consider the annulus $A_l$ with inner radius $r_{l-1}$ and outer radius $r_{l}.$ Then
\begin{align*}
    \mu_{\alpha}\left( A_l\right) & = 2\pi \int_{r_{l-1}}^{r_{l}} \frac{1-\alpha}{\pi \alpha (1-r^2)^2}rdr \\
    & = \frac{1-\alpha}{\alpha} \left[ \frac{r_{l}^2 - r_{l-1}^2}{\left( 1 - r_{l}^2 \right) \left( 1 - r_{l-1}^2 \right)} \right]\\
    & = \frac{1-\alpha}{\alpha} \left[ \frac{ \frac{l^2}{n-m+l^2} - \frac{(l-1)^2}{n-m+(l-1)^2}}{\left( 1 - \frac{(l-1)^2}{n-m+(l-1)^2} \right) \left(1 - \frac{l^2}{n-m+l^2} \right)} \right]\\
    & = \frac{1-\alpha}{\alpha} \left[ \frac{(n-m)(2l-1)}{(n-m)^2} \right]\\
    & = \frac{2l-1}{m}.
\end{align*}
It follows from Lemma \ref{L: expectedval} that the expected number of
eigenvalues in $A_l$ is approximately $2l-1.$ 

\begin{proof}[Proof of Theorem \ref{L: individual}]
The essential idea of the proof is that if $\lambda_p$ is far from its
predicted location $\tilde{\lambda}_p$, then there is either a set of the form
$A_{\ell,\theta}$ with substantially more eigenvalues than predicted
by the mean (if $\lambda_p$ comes early) or a set of the form
$A_{\ell,\theta}$ with substantially fewer eigenvalues than predicted
by the mean (if $\lambda_p$ comes late).  Theorem \ref{L:bernstein}
then gives control  on the probabilities of such events.

To implement this strategy, several cases must be considered, which we
first outline here.
\begin{enumerate}[label=(\Roman*)]
\item $\lambda_p\prec\tilde{\lambda}_p$
\begin{enumerate}[label=(\Alph*)]
\item $\frac{s}{2(l-1)} < \frac{2\pi q}{2l-1}$ 
\item $\frac{2\pi q}{2l-1} \leq \frac{s}{2(l-1)}$ 
\end{enumerate}
\item $\lambda_p\succ\tilde{\lambda}_p$
\begin{enumerate}[label=(\Alph*)]
\item $\frac{s}{2(l-1)} < 2\pi - \frac{2\pi q}{2l-1}$ 
\item $2\pi - \frac{2\pi q}{2l-1} \leq \frac{s}{2(l-1)} \leq \pi$ 
\item $\pi < \frac{s}{2(l-1)} \leq \frac{ \sqrt{m} \left( 1 - \frac{ \epsilon_m}{1-\alpha ( 1- \epsilon_m)} \right)^{\frac{1}{2}} + l-1}{2(l-1)}$  
\item $\frac{ \sqrt{m} \left( 1 - \frac{ \epsilon_m}{1-\alpha ( 1- \epsilon_m)} \right)^{\frac{1}{2}} + l-1}{2(l-1)} < \frac{s}{2(l-1)}$ 
\end{enumerate}
\end{enumerate}
Combining cases (A) and (B) from both I and II  yields the first part of the lemma (small $s$) and combining (C) and (D) of II gives the second part of the lemma (large $s$). 

In most of the cases we will make use of the fact that 
\begin{align}\label{E: dist-in-circle}
    |Re^{i\theta}-re^{i\phi}| \leq ra(\theta, \phi) + |R-r|,
\end{align}
where $a(\theta, \phi)$ denotes the length of the shorter arc on the unit circle between $e^{i\theta}$ and $e^{i\phi}$. \\
\begin{figure}
\begin{center}
\includegraphics[width=0.8\textwidth]{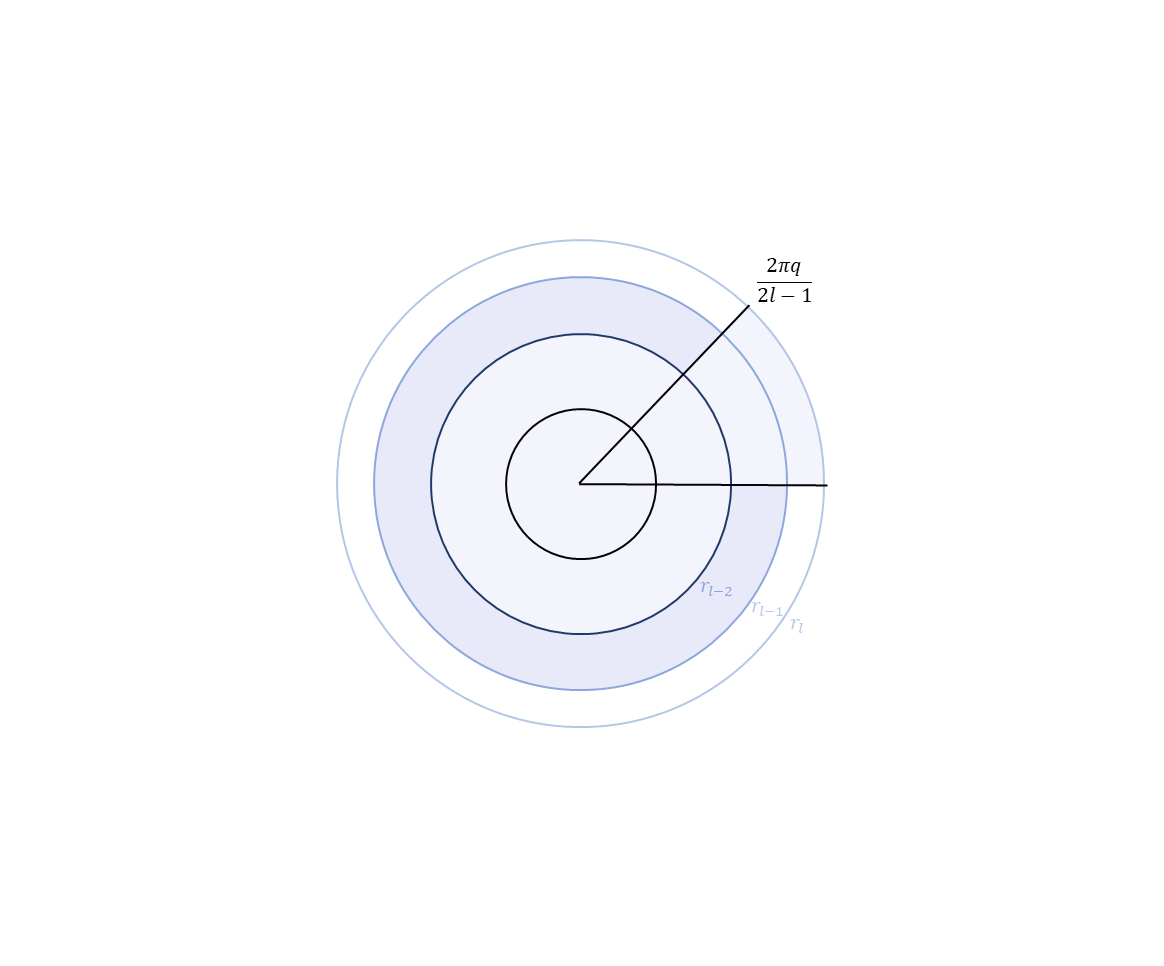}
\end{center}
\caption{$\lambda_p \prec \tilde{\lambda}_p$.} 
\label{fig: prec}
\end{figure}
\noindent {\bf (I, A)} Suppose that $|\lambda_p-\tilde{\lambda}_p|\ge\frac{s}{\sqrt{n-m+(l-1)^2}}$, that $\lambda_p \prec \tilde{\lambda}_p$, and that $\frac{s}{2(l-1)} < \frac{2\pi q}{2l-1}.$ We claim that \begin{align*}
    \lambda_p \prec r_{l-1} \exp \left[ i \left( \frac{2\pi q}{2l-1}-\frac{s}{2(l-1)}\right) \right].
\end{align*}
Indeed, since $\lambda_p \prec \tilde{\lambda}_p$, either
\begin{enumerate}[label= (\roman*)]
\item $ \begin{aligned}[t] r_{l-1} \leq |\lambda_p| < r_{l} &  \text{ and } \arg \lambda_p < \arg \tilde{\lambda}_p = \frac{2\pi q}{2l-1} 
\end{aligned} $
 or 
\item $ \begin{aligned}[t]  |\lambda_p| < |\tilde{\lambda}_p| = r_{l-1} .
\end{aligned}$
\end{enumerate}
If $|\lambda_p| < r_{l-1}$ holds, then the claim holds trivially. Otherwise, the estimate in (\ref{E: dist-in-circle}) implies
\begin{align*}
    |\lambda_p - \tilde{\lambda}_p| & \leq r_{l-1}a\left( \arg \lambda_p, \frac{2\pi q}{2l-1} \right) + | |\lambda_p|- r_{l-1}| \\
    & \leq \frac{l-1}{\sqrt{n-m+(l-1)^2}}a \left( \arg \lambda_p, \frac{2\pi q}{2l-1} \right) + \frac{s}{2\sqrt{n-m+(l-1)^2}}.
\end{align*}
Therefore when condition (i) holds and $|\lambda_p - \tilde{\lambda}_p| \geq \frac{s}{\sqrt{n-m+(l-1)^2}}$, then
\begin{align} \label{E: large mod}
a \left( \arg \lambda_p, \frac{2\pi q}{2l-1} \right) \geq \frac{s}{2(l-1)}
\end{align}
and so $\arg \lambda_p < \frac{2\pi q}{2l-1} - \frac{s}{2(l-1)}.$ In this case as well, then,
\begin{align*}
    \lambda_p \prec r_{l-1} \exp \left[ i \left( \frac{2\pi q}{2l-1}-\frac{s}{2(l-1)}\right) \right].
\end{align*}
It follows from the claim that 
\begin{align*}
    \mathcal{N}_{l-1, \frac{2\pi q}{2l-1}-\frac{s}{2(l-1)} } \geq p.
\end{align*}
Now, the computation of $m \mu_{\alpha} \left( A_{i,\theta}\right)$ in the proof of Corollary \ref{T:segment-means} gives that
\begin{align*}
     m \mu_{\alpha} \left( A_{l-1, \frac{2\pi q}{2l-1}-\frac{s}{2(l-1)} }\right)  &  = (l-1)^2 + q - \frac{s(2l-1)}{4\pi(l-1)}  = p - \frac{s(2l-1)}{4\pi(l-1)} \leq p - \frac{s}{2\pi}.
\end{align*}
Then Theorem \ref{L:bernstein} implies that
\begin{align*}
    \mathbb{P}  & \left[ \mathcal{N}_{l-1, \frac{2\pi q}{2l-1}-\frac{s}{2(l-1)} } \geq p \right] \\
    & \quad \leq \mathbb{P} \left[ \mathcal{N}_{l-1, \frac{2\pi q}{2l-1}-\frac{s}{2(l-1)} } - m\mu_{\alpha} \left( A_{l-1, \frac{2\pi q}{2l-1}-\frac{s}{2(l-1)} }\right) \geq \frac{s}{2\pi} \right] \\
    & \quad \leq 2\exp\left[ - \min \left\{ \frac{s^2}{C_\alpha (l-1)\sqrt{\log(l-1)}}, \frac{s}{8\pi} \right\} \right]\\&=2\exp\left\{-\frac{s^2}{C_\alpha l\sqrt{\log(l)}}\right\},
\end{align*}
since $s\le2\pi(l-1)$.

\noindent{\bf (I, B)} Suppose that $|\lambda_p-\tilde{\lambda}_p|\ge\frac{s}{\sqrt{n-m+(l-1)^2}}$, $\lambda_p \prec \tilde{\lambda}_p$ and $\frac{2\pi q}{2l-1} \leq \frac{s}{2(l-1)}.$ We claim that
\begin{align*}
    \lambda_p \prec \frac{l-2}{\sqrt{n-m+(l-2)^2}}\exp\left[ i \left( 2\pi + \frac{2\pi q}{2l-1} - \frac{s}{2(l-1)} \right) \right].
\end{align*}
The estimate (\ref{E: dist-in-circle}) implies condition (ii) above must hold; that is, $|\lambda_p|< r_{l-1}$. If $|\lambda_p| \geq r_{l-1} - \frac{s}{2\sqrt{n-m+(l-1)^2}}$, then the estimate (\ref{E: dist-in-circle}) again implies that
\begin{align*} 
\pi \geq a \left( \arg \lambda_p, \frac{2\pi q}{2l-1} \right) \geq \frac{s}{2(l-1)}.
\end{align*} 
In particular, when $r_{l-2} \leq |\lambda_p| < r_{l-1}$, $\arg \lambda_p < 2\pi + \frac{2\pi q}{2l-1} - \frac{s}{2(l-1)}$.
If $|\lambda_p|<r_{l-1}- \frac{s}{2\sqrt{n-m+(l-1)^2}} < r_{l-2}$, then the estimate (\ref{E: dist-in-circle}) implies 
\begin{align*}
    |\lambda_p - \tilde{\lambda}_p| & \leq r_{l-1}a\left( \arg \lambda_p, \frac{2\pi q}{2l-1} \right) + | |\lambda_p|- r_{l-1}| \\
    & \leq \frac{l-1}{\sqrt{n-m+(l-1)^2}}a \left( \arg \lambda_p, \frac{2\pi q}{2l-1} \right) + \frac{4(l-1)- s}{2\sqrt{n-m+(l-1)^2}}.
\end{align*}
Then 
\begin{align*}
\frac{s}{2(l-1)} \leq \frac{a \left( \arg \lambda_p, \frac{2\pi q}{2l-1} \right) +2}{3} \leq \pi.
\end{align*}
Either way, $\frac{s}{2(l-1)} \leq \pi$,
\begin{align*}
    \lambda_p \prec \frac{l-2}{\sqrt{n-m+(l-2)^2}}\exp\left[ i \left( 2\pi + \frac{2\pi q}{2l-1} - \frac{s}{2(l-1)} \right) \right],
\end{align*}
and 
\begin{align*}
    \mathcal{N}_{l-2, 2\pi + \frac{2\pi q}{2l-1} -\frac{s}{2(l-1)}}  \geq p.
\end{align*}
Now the computation in the proof of Corollary \ref{T:segment-means} yields
\begin{align*}
   m \mu_{\alpha}\left( A_{l-2, 2\pi + \frac{2\pi q}{2l-1} -\frac{s}{2(l-1)}} \right)
   & = (l-1)^2 + \frac{q(2l-3)}{2l-1} - \frac{s(2l-3)}{4\pi (l-1)} \leq p - \frac{s}{4\pi}
\end{align*}
for $l \geq  2.$ Therefore in this range of $s$, Theorem \ref{L:bernstein} implies that 
\begin{align*} \begin{split}
    \mathbb{P} & \left[ \mathcal{N}_{l-2, 2\pi + \frac{2\pi q}{2l-1} -\frac{s}{2(l-1)}}  \geq p \right] \\
    & \quad \leq \mathbb{P} \left[ \mathcal{N}_{l-2, 2\pi + \frac{2\pi q}{2l-1} -\frac{s}{2(l-1)}}  - m\mu_{\alpha}\left( A_{l-2, 2\pi + \frac{2\pi q}{2l-1} -\frac{s}{2(l-1)}} \right) \geq \frac{s}{4\pi} \right] \\
    & \quad  \leq 2\exp\left[ - \min \left\{ \frac{s^2}{C_\alpha (l-2)\sqrt{\log(l-2)}}, \frac{s}{16\pi} \right\} \right]\\
    & \quad=2\exp\left\{ - \frac{s^2}{C_\alpha l\sqrt{\log(l)}}\right\}.
\end{split} \end{align*}
The estimates above cover the entire range of $s$ when $\lambda_p \prec \tilde{\lambda}_p$ and so
\begin{align*} \begin{split}
    \mathbb{P} & \left[ |\lambda_p - \tilde{\lambda}_p| \geq \frac{s}{\sqrt{n-m + (l-1)^2}}, \lambda_p \prec \tilde{\lambda}_p \right] \le 2\exp\left\{ - \frac{s^2}{C_\alpha l\sqrt{\log(l)}}\right\}
\end{split} \end{align*}
for all $s>0$. \\
\begin{figure}
\begin{center}
\includegraphics[width=0.8\textwidth]{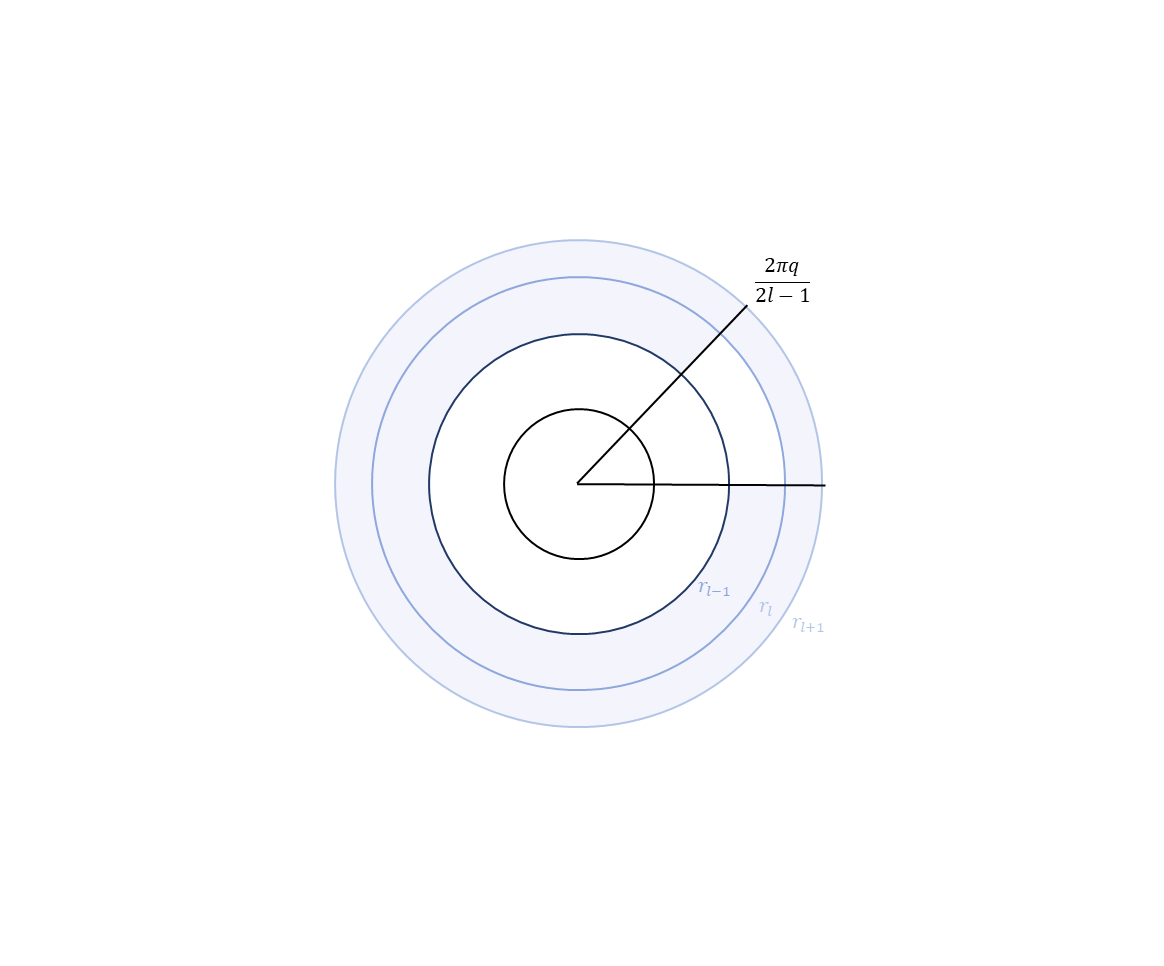}
\end{center}
\caption{$\lambda_p \succ \tilde{\lambda}_p$.} 
\label{fig: succ}
\end{figure}
\noindent {\bf (II, A)} Suppose that $|\lambda_p-\tilde{\lambda}_p|\ge\frac{s}{\sqrt{n-m+(l-1)^2}}$, that $\lambda_p \succ \tilde{\lambda}_p$, and that $\frac{s}{2(l-1)} < 2\pi - \frac{2\pi q}{2l-1}.$ We claim that 
\begin{align*}
    \lambda_p \succ \frac{l-1}{\sqrt{n-m+(l-1)^2}} \exp \left[ i \left( \frac{2\pi q}{2l-1} + \frac{s}{2(l-1)} \right) \right].
\end{align*}
Indeed, since $\lambda_p \succ \tilde{\lambda}_p$, either 
\begin{enumerate}[label=(\roman*)]
\item $\begin{aligned}[t]
    r_{l-1} \leq |\lambda_p| < r_l \text{ and } \arg \lambda_p > \arg \tilde{\lambda}_p 
\end{aligned}$ 
or
\item $\begin{aligned}[t]
    |\lambda_p| \geq r_l = \frac{l}{\sqrt{n-m+l}}.
\end{aligned}$
\end{enumerate}
If $|\lambda_p| \geq r_l$ holds, then the claim is trivially true. Suppose that condition (i) holds.   Then for $s \geq 2$, $|\lambda_p| < r_l < r_{l-1} + \frac{s}{2\sqrt{n-m+(l-1)^2}}$.  As above, combining this observation with the fact that $|\lambda_p - \tilde{\lambda}_p| \geq \frac{s}{\sqrt{n-m+(l-1)^2}}$ and the estimate (\ref{E: dist-in-circle}) implies
$a \left( \arg \lambda_p, \frac{2\pi q}{2l-1} \right) \geq \frac{s}{2(l-1)}.$ It follows that if condition (i) holds, then $\arg \lambda_p > \frac{2\pi q}{2l-1} + \frac{s}{2(l-1)}$, and so
\begin{align*}
    \lambda_p \succ \frac{l-1}{\sqrt{n-m+(l-1)^2}} \exp \left[ i \left( \frac{2\pi q}{2l-1} + \frac{s}{2(l-1)} \right) \right].
\end{align*}
It follows from the claim that \begin{align*}
    \mathcal{N}_{l-1, \frac{2\pi q}{2l-1}+\frac{s}{2(l-1)}} < p.
\end{align*}
By the proof of Corollary (\ref{T:segment-means}),
\begin{align*}
    m \mu_{\alpha}\left( A_{l-1, \frac{2\pi q}{2l-1}+\frac{s}{2(l-1)}} \right)
    & = p + \frac{s}{2\pi} \frac{2l-1}{2l-2} \geq p + \frac{s}{2\pi},
\end{align*}
and so Theorem \ref{L:bernstein} implies that
\begin{align*}
    \mathbb{P} & \left[ \mathcal{N}_{l-1, \frac{2\pi q}{2l-1}+\frac{s}{2(l-1)}}  <p \right] \\
    & \quad  \leq \mathbb{P} \left[ m\mu_{\alpha}\left( A_{l-1, \frac{2\pi q}{2l-1}+\frac{s}{2(l-1)}} \right) - \mathcal{N}_{l-1, \frac{2\pi q}{2l-1}+\frac{s}{2(l-1)}}   > \frac{s}{2\pi} \right] \\
    & \quad  \leq 2\exp\left[ - \min \left\{ \frac{s^2}{C_\alpha(l-1)\sqrt{\log(l-1)}}, \frac{s}{8\pi} \right\} \right]\\&\quad=2\exp\left[ - \frac{s^2}{C_\alpha(l-1)\sqrt{\log(l-1)}}\right],
\end{align*}
since $s\le 2\pi(l-1)$.

\noindent {\bf (II, B)} Suppose that $|\lambda_p-\tilde{\lambda}_p|\ge\frac{s}{\sqrt{n-m+(l-1)^2}}$, that $\lambda_p \succ \tilde{\lambda}_p$ and  that $2\pi - \frac{2\pi q}{2l-1} \leq \frac{s}{2(l-1)} \leq \pi.$ We claim that
\begin{align*}
    \lambda_p \succ \frac{l}{\sqrt{n-m+l^2}} \exp \left[ i \left( \frac{2\pi q}{2l-1}+ \frac{s}{2(l-1)} - 2\pi \right) \right].
\end{align*}
By the estimate (\ref{E: dist-in-circle}) again, it must be the case that condition (ii) holds and $|\lambda_p| \geq r_l$. If $s \geq 2$ and $|\lambda_p| \geq r_{l-1} + \frac{s}{2\sqrt{n-m+(l-1)^2}} > r_{l+1}$, then the claim holds trivially. If $|\lambda_p|<r_{l-1} + \frac{s}{2\sqrt{n-m+(l-1)^2}}$, then the estimate (\ref{E: dist-in-circle}) yields $\arg \lambda_p > \frac{2\pi q}{2l-1} + \frac{s}{2(l-1)} - 2\pi.$ In particular, if $r_l \leq |\lambda_p| < r_{l+1}$, then $\arg \lambda_p > \frac{2\pi q}{2l-1} + \frac{s}{2(l-1)} - 2\pi$ and so
\begin{align*}
    \lambda_p \succ \frac{l}{\sqrt{n-m+l^2}} \exp \left[ i \left( \frac{2\pi q}{2l-1}+ \frac{s}{2(l-1)} - 2\pi \right) \right].
\end{align*}
It follows from the claim that
\begin{align*}
    \mathcal{N}_{l,  \frac{2\pi q}{2l-1}+ \frac{s}{2(l-1)} - 2\pi} < p.
\end{align*}
Since
\begin{align*}
    m \mu_{\alpha}\left( A_{l,  \frac{2\pi q}{2l-1}+ \frac{s-2}{l-1} - 2\pi}\right) & = l^2 + \frac{ \frac{2\pi q}{2l-1} + \frac{s}{2(l-1)} - 2\pi}{2\pi} \left(2l + 1 \right) \geq p + \frac{s}{2\pi},
\end{align*}
Theorem \ref{L:bernstein} implies that in this regime,
\begin{align*}
    \mathbb{P} & \left[ \mathcal{N}_{l,  \frac{2\pi q}{2l-1}+ \frac{s}{2(l-1)} - 2\pi} < p \right] \\
    & \quad  \leq \mathbb{P} \left[ m\mu_{\alpha} \left( A_{l,  \frac{2\pi q}{2l-1}+ \frac{s}{2(l-1)} - 2\pi}\right) - \mathcal{N}_{l,  \frac{2\pi q}{2l-1}+ \frac{s}{2(l-1)} - 2\pi} > \frac{s}{2\pi} \right] \\
    &  \quad \leq 2 \exp\left[ - \min \left\{ \frac{s^2}{C_\alpha l\sqrt{\log(l)}}, \frac{s}{8\pi} \right\} \right]\\&\quad=2 \exp\left[ - \frac{s^2}{C_\alpha l\sqrt{\log(l)}}\right].
\end{align*}
Combining cases (I, A), (I,B), (II, A), and (II, B) thus yields 
\begin{align*}
    \mathbb{P}\left[ |\lambda_p - \tilde{\lambda}_p| \geq \tfrac{s}{\sqrt{n-m+(l-1)^2}} \right] &  \leq  2\exp\left[ - \frac{s^2}{C_\alpha l\sqrt{\log(l)}}\right].
\end{align*} 
when $s \leq 2\pi(l-1)$. This proves the first part of the Theorem (small $s$). 

Finally, we consider cases for the larger  values of $s$ based on which part of Theorem \ref{L:bernstein} applies. \\ 

\noindent {\bf (C)} Let $\epsilon_m=\sqrt{\frac{2\log (m+1)}{m}}$ and suppose that $|\lambda_p-\tilde{\lambda}_p|\ge\frac{s}{\sqrt{n-m+(l-1)^2}}$, that $\lambda_p \succ \tilde{\lambda}_p$, and that
\begin{align*}
2\pi(l-1) \leq s\leq \left[\sqrt{m} \left( 1 - \frac{ \epsilon_m}{1 - \alpha\left( 1- \epsilon_m \right)} \right)^{\frac{1}{2}}+ l-1\right].
\end{align*}
(That is, $s-l + 1 \leq \sqrt{m} \left( 1 - \frac{ \epsilon_m}{1 - \alpha\left( 1- \epsilon_m \right) } \right)^{\frac{1}{2}}$.) By the triangle inequality,
\begin{align*}
|\lambda_p| \geq \frac{s}{\sqrt{n-m+(l-1)^2}} - |\tilde{\lambda}_p| = \frac{s-l+1}{\sqrt{n-m+(l-1)^2}}.
\end{align*}
Now, $\frac{s}{2(l-1)} > \pi$ implies $(l-1)^2 \leq (s-l+1)^2$. It follows that
$|\lambda_p| \geq \frac{s-l+1}{\sqrt{n-m+(s-l+1)^2}}$, and so
\begin{align*}
\mathcal{N}_{\lfloor s-l+1\rfloor, 2\pi} < p.
\end{align*} 
Since $p = (l-1)^2 +q$, $1 \leq q \leq 2l-1$, and $l < \frac{s}{2\pi} + 1$,
\begin{align*}
m\mu_{\alpha} \left( A_{\lfloor s-l+1\rfloor, 2\pi} \right)  & = (\lfloor
                                                s-l+1\rfloor)^2 +
                                                2\lfloor s-l+1\rfloor + 1 \\
& \ge s^2 - 2sl+l^2+2s-2l+1\\&=s^2-2s(l-1)+p-q\ge cs^2+p,
\end{align*}
since $s\ge 2\pi$.
It follows from Theorem  \ref{L:bernstein},
\begin{align*}
\mathbb{P}  \left[ \mathcal{N}_{\lfloor s-l+1\rfloor, 2\pi} < p \right] &  \leq \mathbb{P} \left[ m\mu_{\alpha} \left( A_{\lfloor s-l+1\rfloor, 2\pi} \right)  - \mathcal{N}_{\lfloor s-l+1\rfloor, 2\pi}   > cs^2\right] \\
&  \leq 2 \exp \left[ - \min \left\{ \frac{c^2s^4}{C_\alpha
  (\lfloor s-l+1\rfloor)\sqrt{\log(\lfloor s-l+1\rfloor)}}, \frac{cs^2}{2} \right\} \right]\\&\le 2
                                                                \exp
                                                                \left[
                                                                -c_\alpha
                                                                s^2\right],
\end{align*}
since $s\ge 2\pi(l-1)$. \\
\noindent {\bf (D)} Suppose that $|\lambda_p-\tilde{\lambda}_p|\ge\frac{s}{\sqrt{n-m+(l-1)^2}}$, that $\lambda_p \succ \tilde{\lambda}_p$, and that
\begin{align*}
\frac{\sqrt{m} \left( 1 - \frac{ \epsilon_m}{1 - \alpha\left( 1 - \epsilon_m \right)} \right)^{\frac{1}{2}}+ l-1}{2(l-1)} < \frac{s}{2(l-1)};
\end{align*}
that is,
\begin{align*}
\sqrt{m} \left( 1 - \frac{ \epsilon_m}{1 - \alpha\left( 1 - \epsilon_m \right)} \right)^{\frac{1}{2}} < s-l+1.
\end{align*}
As in the previous case, 
\begin{align*}
\mathcal{N} \left( A_{\lfloor s-l+1\rfloor, 2\pi} \right) < p.
\end{align*} 
If $\frac{s-l+1}{\sqrt{n-m+(s-l+1)^2}} \geq 2$, then $\mathbb{P} \left[ \mathcal{N} \left( A_{\lfloor s-l+1\rfloor, 2\pi} \right) < p \right] =0$. Otherwise, by the second estimate in Lemma \ref{L: expectedval},
\begin{align*}
\mathbb{P} & \left[ \mathcal{N} \left( A_{\lfloor s-l+1\rfloor, 2\pi} \right) < p \right] \\
&\leq \mathbb{P} \left[ \mathbb{E} \mathcal{N} \left( A_{\lfloor s-l+1\rfloor, 2\pi} \right) - \mathcal{N} \left( A_{\lfloor s-l+1\rfloor, 2\pi} \right) > m \mu_{\alpha} \left( A_{\lfloor s-l+1\rfloor,2\pi} \right) - \tfrac{6\sqrt{2m \log(m+1)}}{1-\alpha} - p \right] \\
&  \leq \mathbb{P} \left[ \mathbb{E} \mathcal{N} \left( A_{\lfloor s-l+1\rfloor, 2\pi} \right) - \mathcal{N} \left( A_{\lfloor s-l+1\rfloor, 2\pi} \right) >cs^2 -  \tfrac{6\sqrt{2m \log(m+1)}}{1-\alpha} \right].
\end{align*}
In this range, $s^2 \geq m \left( 1 - \frac{\epsilon_m}{1 - \alpha(1-\epsilon_m)} \right)$, so the lower bound can be replaced, for large enough $m$, by $cs^2$ by slightly reducing the value of $c$.  Theorem \ref{L:bernstein} applied with $t=cs^2\ge\frac{12\sqrt{2m\log(m)}}{1-\alpha}$ (again for $m$ large enough) then yields
\begin{align*}
\mathbb{P}  \left[ \mathcal{N} \left( A_{\lfloor s-l+1\rfloor, 2\pi} \right) < p \right] &\leq 2 \exp \left[ - \min \left\{ \frac{c^2s^4}{C_\alpha(\lfloor s-l+1\rfloor)\sqrt{\log(\lfloor s-l+1\rfloor)}}, \frac{cs^2}{2} \right\} \right]\\&\le 2 \exp \left[ -c_\alpha
                                                                s^2\right].
\end{align*}
Cases (C) and (D) thus yield
\begin{align*}
    \mathbb{P}\left[ |\lambda_p - \tilde{\lambda}_p| \geq \tfrac{s}{\sqrt{n-m+(l-1)^2}} \right]  & \leq  2 \exp \left[  -c_\alpha
                                                                s^2\right]
\end{align*}
for $s\ge 2\pi(l-1)$.  
Finally, the empirical spectral measure is supported on the disc of radius 1. It follows that if $s \geq 2\sqrt{n-m+(l-1)^2}$, then 
\begin{align*}
    \mathbb{P}\left[ |\lambda_p - \tilde{\lambda}_p| \geq \tfrac{s}{\sqrt{n-m+(l-1)^2}} \right] = 0.
    \end{align*}
This completes the proof. 
\end{proof}

\begin{proof}[Proof of Corollary \ref{T:location-variance}]
Let $\epsilon_m = \sqrt{\frac{2 \log (m+1)}{m}}$ and $p$ be such that 
\begin{align*}
2\le l=\ceil{\sqrt{p}} \leq \sqrt{m}\left( 1 - \frac{ \epsilon_m}{1-\alpha ( 1- \epsilon_m)} \right)^{\frac{1}{2}}. 
\end{align*}
Then by Fubini's theorem and Theorem \ref{L: individual},
\begin{align*}
\var(\lambda_p) & \leq \mathbb{E} \left| \lambda_p - \tilde{\lambda}_p \right|^2 \\
& = \int_0^{\infty} 2t \mathbb{P} \left[ \left| \lambda_p - \tilde{\lambda}_p \right| > t \right] dt \\
& = \frac{2}{n-m+(l-1)^2} \int_0^{\infty} s \mathbb{P} \left[ \left| \lambda_p - \tilde{\lambda}_p \right| > \frac{s}{\sqrt{n-m+(l-1)^2}} \right] ds \\
& \leq \tfrac{2}{n-m+(l-1)^2} \bigg[ \int_0^{2\pi (l-1)} 2s e^{- \frac{s^2}{C_\alpha l\sqrt{\log(l)} }} ds + \int_{2\pi (l-1)}^{2\sqrt{n-m+(l-1)^2}} 2se^{- c_\alpha s^2} ds \bigg] \\
\\
& \leq \tfrac{2}{n-m+(l-1)^2} \bigg[ \int_0^{\infty} 2s e^{- \frac{s^2}{C_\alpha l\sqrt{\log(l)} }} ds + \int_{2\pi (l-1)}^{\infty} 2se^{- c_\alpha s^2}ds \bigg]\\
& \leq \tfrac{C_\alpha}{n-m + (l-1)^2} \left[ l\sqrt{\log(l+1)}\right]\\&\le C_\alpha\frac{l\sqrt{\log(l+1)}}{n},
\end{align*}
since $(l-1)^2\le m$.
\end{proof}

\bibliographystyle{plain}
\bibliography{eigenvalue-rigidity}

\end{document}